\documentclass[11pt]{amsart}

\usepackage{amsmath,amsthm,amsfonts,latexsym,epsfig}

\newtheorem{thm}{Theorem}[section]
\newtheorem{lem}[thm]{Lemma}
\newtheorem{pro}[thm]{Proposition}
\newtheorem{exa}[thm]{Example}

\newcommand{\al}{\alpha}
\newcommand{\be}{\beta}
\newcommand{\ga}{\gamma}
\newcommand{\si}{\sigma}

\newcommand{\ta}{\theta}
\newcommand{\Ta}{\Theta}

\newcommand{\la}{{\lambda}}
\newcommand{\Ga}{\Gamma}

\newcommand{\ep}{\epsilon}

\newcommand{\cP}{\mathcal{P}}

\newcommand{\cB}{\mathcal{B}}

\newcommand{\cM}{\mathcal{M}}
\newcommand{\fS}{\mathfrak{S}}

\newcommand{\bbZ}{\mathbb{Z}}

\newcommand{\del}{\partial}

\newcommand{\bfz}{{\bf z}}
\newcommand{\bfy}{{\bf y}}

\newcommand{\bfp}{{\bf p}}
\newcommand{\bfq}{{\bf q}}
\newcommand{\bfw}{{\bf w}}
\newcommand{\bfu}{{\bf u}}

\newcommand{\dmrjdel}[1]{}

\title{Diagonal Solutions to the 2-Toda Hierarchy}
\author[S.~R.~Carrell]{S.~R.~Carrell$^*$}

\thanks{
${\hspace{-1ex}}^*$Department of Combinatorics and Optimization,
                                                University of Waterloo, Waterloo, Ontario, Canada.;    \\
${\hspace{.35cm}}$ \texttt{srcarrell@uwaterloo.ca}}

\begin{document}

\begin{abstract}
	Using a combinatorial description of the Bernstein operator and its action on Schur functions, we describe the formal power series solutions
	to a family of partial differential equations known as the 2-Toda hierarchy. We also characterize diagonal solutions and use this to prove
	that a special family of formal power series, called content-type series, are solutions to the 2-Toda hierarchy. As examples, we prove that various
	generating series for permutation factorization problems (including the double Hurwitz problem), correlators for the Schur measure on partitions
	and the formal character expansion of the HCIZ integral are all solutions to the 2-Toda hierarchy.
\end{abstract}

\maketitle

\section{Introduction}

Integrable hierarchies, certain families of partial differential equations, have found a wide range of applications in areas such
as random matrices, enumerative combinatorics and stochastic models. One of the prominent examples of this are the numerous
problems whose generating function is a solution to the 2-Toda hierarchy (and the related KP hierarchy). Curiously, most of these
series can be expressed as content-type series similar to those introduced in \cite{GJ1}.

The connection between integrable hierarchies and random matrix models has been known for some time and has been used to derive
properties of various statistical quantities of interest. Some examples of this include the results of Adler and Moerbeke \cite{AM3, AM2}
on the spectrum of random matrices as well as the results of Orlov, Yu and Shcherbin \cite{Orlov1, Orlov2, Orlov3, OS1} relating
various iterated integrals to the KP and 2-Toda hierarchies. Related to this is the Fredholm determinant approach to random matrix
models used, for example, by Tracy and Widom \cite{TW3, TW2, TW1}. Although the precise relationship between the Fredholm determinant
approach and the integrable hierarchies discussed here is not fully known, some progress has been made \cite{Ra1}.

One of the more well known applications of integrable hierarchies to enumerative combinatorics is Okounkov's result \cite{Ok1} that the
generating series for double Hurwitz numbers satisfies the 2-Toda hierarchy, thus settling a conjecture of Pandharipande \cite{Pand1} concerning
Gromov Witten theory of the sphere. This was then used by Kazarian \cite{Kaz1} in one of the
proofs of Witten's conjecture on the generating function of linear Hodge integrals.

Another application of integrable hierarchies to an enumerative problem is the result of Goulden and Jackson \cite{GJ1} that the number
of rooted triangulations (with respect to number of vertices and genus) satisfies a quadratic recurrence. Bender, Gao and Richmond \cite{BGR1}
then used this to derive a recurrence for the map asymptotics constant. A related result was obtained by Ercolani \cite{E1} using Riemann-Hilbert
techniques.

An example of the connection between various stochastic models and integrable hierachies is through the Schur measure introduced by
Okounkov \cite{Ok2} and its correlators. The Schur measure is a generalization of the $z$-measure introduced by Borodin and Olshanksi \cite{BO1, BO2, BO3}
in the context of asymptotic representation theory. In addition, it has been shown \cite{Bor1} that the $z$-measure encodes a number of stochastic
models considered, for example, by Johansson \cite{Jo1}.

In this paper we show that the general content-type series mentioned above are formal power series solutions to the 2-Toda hierarchy
(for a similar result coming from a different approach see \cite{Orlov1, Orlov2, Orlov3, OS1}). Moreover, we prove a partial converse.
That is, we show that if one has a formal power series solution to the 2-Toda hierarchy that also satisfies some constraints on
its coefficients then it can be written as a content-type series.

The outline of this paper is as follows. In section 2 we discuss notation and recall some results from previous work, \cite{CG1},
concerning the action of the Bernstein operator on Schur polynomials. We then describe the KP hierarchy and the 2-Toda hierarchy,
including a characterization of formal power series solutions with respect to their Schur polynomial coefficients.

In section 3 we first specialize the characterization of formal power series solutions of the 2-Toda hierarchy to diagonal
solutions. We then introduce the content-type series and, using the specialized characterization, we prove our first main
result that the content-type series are solutions to the 2-Toda hierarchy. We finish the section with our second main
result, the partial converse.

In the final section we discuss a few different examples of content-type series that appear in the literature. We start by discussing
the solution of an enumeration problem which was discussed in \cite{GJ1} concerning certain tuples of permutations. In addition, we
recall a specialization of this result which gives an alternate proof of the fact that the double Hurwitz series can be embedded in a solution
to the 2-Toda hierarchy. We then discuss the Schur measure on partitions and show that its correlators satisfy the 2-Toda hierarchy.
Lastly, we show that the Harish-Chandra Itzykson Zuber integral also satisfies the 2-Toda hierarchy.

\section{Background}

We begin with some notation for partitions (for additional information see \cite{Mac}). If $\la_1, \cdots, \la_n$ are integers
with $\la_1 \geq \cdots \geq \la_n \geq 1$ and $\la_1 + \cdots + \la_n = d$, then $\la = (\la_1, \cdots, \la_n)$ is said to
be a partition of $|\la| := d$ with $\ell(\la) := n$ parts. The empty list $\ep$ of integers is the unique partition with $d = n = 0$.
We will use $\cP$ to denote the set of all partitions. If $\la$ has $f_j$ parts equal to $j$ for $j = 1, \cdots, d$, then we may also
write $\la = d^{f_d} \cdots 1^{f_1}$ and if $f_j = 1$ for some $j$ we will omit the exponent. Also, Aut $\la$ denotes the set of permutations
that fix $\la$; therefore $|\mbox{ Aut } \la| = \prod_{j \geq 1} f_j!$. We will often identify a partition $\la$
with its diagram, a left justified array of unit squares, called cells, with $\la_i$ cells in the $i$th row. The conjugate of $\la$ is the partition
$\la^t$ whose diagram is the diagram of $\la$, reflected along the main diagonal. We will use $\square$ to denote a cell and we will
write $\square \in \la$ to mean that $\square$ is a cell of the partition $\la$. Also, for $\square \in \la$ we define the content of the cell $\square$
by $c(\square) := j - i$ where $j$ is the column index of $\square$ and $i$ is the row index.

Now, let $\la \in \cP$ and $i \geq 1$ be an integer. Let $u_i(\la)$ be the unique integer such that
	\[ \la_{u_i(\la)} \geq i > \la_{u_i(\la) + 1}. \]
We define
	\[ \la \uparrow i = (\la_1 - 1, \la_2 - 1, \cdots, \la_{u_i(\la)} - 1, i - 1, \la_{u_i(\la)+1}, \cdots ). \]
The operation $\la \mapsto \la \uparrow i$ can be thought of informally as follows. We add a part of size $i$ to $\la$ such that the result is
still a partition and the index of the added part is as large as possible. We then reduce the size of each part not displaced (including the
new part) by one. From the definition of $\la \uparrow i$, we compute
	\[ |\la \uparrow i| = |\la| + i - u_i(\la) - 1, \]
and
	\[ |\la| - \ell(\la) = |\la \uparrow 1| < |\la \uparrow 2| < \cdots. \]
\begin{exa}
	Suppose $\la = 754^21$ and that we wish to determine $\la \uparrow 4$. The largest index at which $4$ can be placed in $\la$ such that the result is
	still a partition is $5$ and so $u_4(\la) = 5$. Thus $\la \uparrow 4 = 643^31$. Similarly, if we wish to determine $\la \uparrow 3$ we have
	$u_3(\la) = 5$ again and so $\la \uparrow 3 = 643^221$.
	
	Now suppose that $\la = \ep$. Then for any $i \geq 1$ we have $u_i(\ep) = 1$ and so $\ep \uparrow i = i - 1$. If $\la = 1^k$ then we have
	$u_1(1^k) = k + 1$ and so $1^k \uparrow 1 = \ep$.
\end{exa}

We also define the dual operation. Let $\la \in \cP$ and $j \geq 1$ be an integer. We define
	\[ \la \downarrow j = (\la_1 + 1, \la_2 + 1, \cdots, \la_{j-1} + 1, \la_{j+1}, \cdots). \]
Informally we think of the operation $\la \mapsto \la \downarrow j$ as removing the part $\la_j$ and then increasing the size of each part not displaced.
From the definition of $\la \downarrow j$, we compute
	\[ |\la \downarrow j| = |\la| + j - \la_j - 1, \]
and
	\[ |\la| - \la_1 = |\la \downarrow 1| < |\la \downarrow 2| < \cdots. \]
\begin{exa}
	Suppose $\la = 643^31$. Then we easily see that $\la \downarrow 5 = 754^21$. Similarly, if $\la = 643^221$ then $\la \downarrow 5 = 754^21$.
	
	Now suppose that $\la = \ep$. Then for any $j \geq 1$ we see that $\ep \downarrow j = 1^{j-1}$. If $\la = k$ then it is easily seen that
	$\la \downarrow 1 = \ep$.
\end{exa}

Additionally, from the definition of $\la \uparrow i$ and $\la \downarrow j$, we have
	\[ (\la \uparrow i) \downarrow (u_i(\la) + 1) = \la, \]
and
	\[ (\la \downarrow j) \uparrow (\la_j + 1) = \la. \]

There is another, more combinatorial, description of the operations $\la \mapsto \la \uparrow i$ and $\la \mapsto \la \downarrow j$ which was discussed
in \cite{CG1}, however, we shall not have need of it here. It does, however, make the following relationship clear,
	\[ (\la^t) \uparrow i = (\la \downarrow i)^t. \]
	
Throughout this paper we use $t$, $a$, $b$, $\bfp = (p_1, p_2, \cdots)$, $\bfq = (q_1, q_2, \cdots)$, $\bfw = (w_1, w_2, \cdots)$, $\bfz = (z_1, z_2, \cdots)$,
and $\bfy = (\cdots, y_{-1}, y_0, y_1, \cdots)$ to denote algebraically independent indeterminates. We also write $\bfp + \bfq$ to mean the
sequence $(p_1 + q_1, p_2 + q_2, \cdots)$. For $\la \in \cP$ we write $p_\la = \prod_{i \geq 1} p_{\la_i}$.

For $i \geq 0$ define the polynomials $h_i(\bfp)$ by
	\[ \sum_{i \geq 0} h_i(\bfp) t^i = \exp\left( \sum_{k \geq 1} \frac{p_k}{k} t^k \right), \]
and for $i < 0$, $h_i(\bfp) = 0$.

For $\la \in \cP$ define the polynomials $s_\la(\bfp)$ by
	\[ s_\la(\bfp) = \det\left( h_{\la_i - i + j}(\bfp) \right)_{1 \leq i,j \leq n}, \]
where $n \geq \ell(\la)$. The $s_\la(\bfp)$ are called the Schur polynomials and they form a basis for the ring of formal power series in $\bfp$.

We introduce the inner product $\langle \cdot, \cdot \rangle$ defined by
	\[ \langle s_\la(\bfp), s_\mu(\bfp) \rangle = \delta_{\la,\mu}, \,\, \forall \la, \mu \in \cP. \]
For any polynomial $f(\bfp)$ we define the adjoint of multiplication by $f$, written $f^\perp$, by requiring that for all polynomials $g(\bfp)$ and $h(\bfp)$,
	\[ \langle f^\perp(\bfp) g(\bfp), h(\bfp) \rangle = \langle g(\bfp), f(\bfp) h(\bfp) \rangle. \]
It can be shown (see \cite{Mac} Chapter I, Section 5, Exercise 3) that $p_i^\perp = i \frac{\del}{\del p_i}$ and that for any polynomial
$f(\bfp)$, $f^\perp = f(p_1^\perp, p_2^\perp, \cdots)$. In particular, $f^\perp$ is a differential operator.

Note that if we interpret $p_i$ as the $i$th power sum symmetric function then the Schur polynomials, $s_\la$, become the Schur symmetric functions and
$\langle \cdot, \cdot \rangle$ becomes the Hall inner product\cite{Mac}.

We now define the Bernstein operator,
	\[ B(\bfp; t) = \exp\left( \sum_{k \geq 1} \frac{t^k}{k} p_k \right) \exp\left( -\sum_{k \geq 1} \frac{t^{-k}}{k} p_k^\perp \right), \]
and the adjoint Bernstein operator
	\[ B^\perp(\bfp; t) = \exp\left( -\sum_{k \geq 1} \frac{t^k}{k} p_k \right) \exp\left( \sum_{k \geq 1} \frac{t^{-k}}{k} p_k^\perp \right). \]

In \cite{CG1} we showed that the Bernstein operator and its adjoint act nicely on Schur polynomials. Specifically, we have the following
\begin{thm}\label{bernsteinaction}
	Suppose that $a_\la, \la \in \cP$ are scalars. Then
		\[ B(\bfp; t) \sum_{\la \in \cP} a_\la s_\la(\bfp)
					= \sum_{\be \in \cP} s_\be(\bfp) \sum_{k \geq 1} (-1)^{k-1} t^{|\be| - |\be \downarrow k|} a_{\be \downarrow k}, \]
	and
		\[ B^\perp(\bfp; t) \sum_{\la \in \cP} a_\la s_\la(\bfp)
					= \sum_{\al \in \cP} s_\al(\bfp) \sum_{m \geq 1} (-1)^{|\al| - |\al \uparrow m| + m - 1} t^{|\al| - |\al \uparrow m|} a_{\al \uparrow m}. \]
\end{thm}

In addition to showing that the Bernstein operator acts nicely on Schur polynomials, we also showed that the commutator of $B(\bfp; t)$ and a certain
translation operator was simple. Specifically, define
	\[ \Ta(\bfp, \bfq) = \exp\left( \sum_{k\geq 1} q_k \frac{\del}{\del p_k} \right), \]
and
	\[ \Ga(\bfq; t) = \exp\left( \sum_{k \geq 1} \frac{t^i}{i} q_i \right). \]
Using multivariate Taylor series, we see that if $f(\bfp)$ is a formal power series then
	\[ \Ta(\bfp, \bfq) f(\bfp) = f(\bfp + \bfq). \]
It is not difficult to show that the following relations hold between $B$ and $\Ta$. Note that this result essentially arises in the proof of
Theorem 5.3 in \cite{CG1}.
\begin{pro}\label{commutator}
	We have
		\[ B(\bfp; t) \Ta(\bfp, \bfq) = \Ga(\bfq; t)^{-1} \Ta(\bfp, \bfq) B(\bfp; t), \]
	and
		\[ B^\perp(\bfp; t) \Ta(\bfp, \bfq) = \Ga(\bfq; t) \Ta(\bfp, \bfq) B^\perp(\bfp; t). \]
\end{pro}

We can use Theorem~\ref{bernsteinaction} and Proposition~\ref{commutator} to describe various families of partial differential equations.
For the purposes of this paper we will focus on the 2-Toda hierarchy although other related integrable hierarchies such as the $n$-KP
hierarchy \cite{KvL1} can be treated similarly. It is also likely that similar results can be obtained in the B-type case using the results
in \cite{HJS1}, however we have not done so here.

As an easier example, we begin with a description of the KP hierarchy. Suppose that $\tau(\bfp)$ is a formal power series. Then $\tau(\bfp)$ is
a solution to the KP hierarchy if and only if
	\[ [t^{-1}] \left( B(\bfp; t) \tau(\bfp) \right) \left( B^\perp(\bfq; t) \tau(\bfq) \right) = 0. \]
Here we use square brackets to denote the coefficient extraction operator.

\begin{thm}\label{kphierarchy}
	Suppose that $a_\la, \la \in \cP$ are scalars and that
		\[ \tau(\bfp) = \sum_{\la \in \cP} a_\la s_\la(\bfp). \]
	The following are equivalent.
	\begin{enumerate}
		\item[(i)] The formal power series $\tau(\bfp)$ is a solution to the KP hierarchy.
		\item[(ii)] The formal power series $\tau(\bfp + \bfq)$ is a solution to the KP hierarchy in the variables $\bfp$.
		\item[(iii)] For all $\al, \be \in \cP$,
										\[ \sum_{i,j} (-1)^{|\al| - |\al \uparrow i| + i + j} a_{\al \uparrow i} a_{\be \downarrow j} = 0, \]
								 where the sum is over all integers $i,j \geq 1$ such that $|\al \uparrow i| + |\be \downarrow j| = |\al| + |\be| + 1$.
		\item[(iv)] For all $\al, \be \in \cP$,
										\[ \sum_{i,j} (-1)^{|\al| - |\al \uparrow i| + i + j} (s^\perp_{\al \uparrow i}(\bfp) \tau(\bfp)) 
																																													(s^\perp_{\be \downarrow j}(\bfp) \tau(\bfp)) = 0, \]
								 where the sum is over all integers $i,j \geq 1$ such that $|\al \uparrow i| + |\be \downarrow j| = |\al| + |\be| + 1$.
	\end{enumerate}
\end{thm}
\begin{proof}
	The fact that (ii) implies (i) follows immediately by setting $q_i = 0$ for all $i \geq 1$. The fact that (i) implies (ii) follows from
	the definition of the KP hierarchy and Proposition~\ref{commutator}. That (i) and (iii) are equivalent follows by taking coefficients
	and using Theorem~\ref{bernsteinaction}. Lastly, that (ii) and (iv) are equivalent follows by coefficient extraction, Theorem~\ref{bernsteinaction}
	and the fact that $[s_\la(\bfp)] \tau(\bfp+\bfq) = s_\la^\perp(\bfp)\tau(\bfp)$ (since the Schur polynomials are orthonormal).
\end{proof}

For the 2-Toda hierarchy we have a similar result, however, it is a little more technical since instead of having a single family of indeterminates,
we have two families of indeterminates and a discrete parameter.

A sequence of formal power series $\{ \tau_n(\bfp, \bfq) \}_{n \in \bbZ}$ is a solution to the 2-Toda hierarchy if and only if for all $k, m \in \bbZ$,
	\begin{align*}
		[t^{k-m}]\left( B(\bfp;t) \right. & \left. \tau_{m}(\bfp,\bfq)\right) \left( B^\perp(\bfw;t)\tau_{k+1}(\bfw,\bfz)\right) = \\
				& [t^{m-k}]\left( B^\perp(\bfq;t)\tau_{m+1}(\bfp,\bfq)\right) \left(B(\bfz;t)\tau_{k}(\bfw;\bfz)\right).
	\end{align*}
\begin{thm}\label{2todahierarchy}
	Suppose that $a^\la_\mu(n), \la,\mu \in \cP$, $n \in \bbZ$ are scalars and that for all $n \in \bbZ$,
		\[ \tau_n(\bfp, \bfq) = \sum_{\la,\mu \in \cP} a^\la_\mu(n) s_\la(\bfp) s_\mu(\bfq). \]
	The following are equivalent.
	\begin{enumerate}
		\item[(i)]	The sequence of formal power series $\{ \tau_n(\bfp,\bfq) \}_{n \in \bbZ}$ is a solution to the 2-Toda hierarchy.
		\item[(ii)] The sequence of formal power series $\{ \tau_n(\bfp+\bfw,\bfq+\bfz) \}_{n \in \bbZ}$ is a solution to the 2-Toda hierarchy
								in the variables $\bfp$ and $\bfq$.
		\item[(iii)]	For all $m, k \in \bbZ$ and $\al, \be, \la, \mu \in \cP$,
										\begin{align*}
											\sum_{i,j} (-1)^{|\al|-|\al\uparrow j|+i+j} & a_\mu^{\la\downarrow i}(m) a_\be^{\al\uparrow j}(k+1) = \\
													& \sum_{s,t} (-1)^{|\mu| + |\mu\uparrow s|+s+t} a^\la_{\mu\uparrow s}(m+1) a^\al_{\be\downarrow t}(k),
										\end{align*}
									where the first sum is over integers $i,j \geq 1$ such that $|\la\downarrow i| + |\al\uparrow j| = |\la| + |\al| + m - k$ and the second sum
									is over integers $s,t \geq 1$ such that $|\mu\uparrow s| + |\be\downarrow t| = |\mu| + |\be| + k - m$.
		\item[(iv)]		for all $m, k \in \bbZ$ and $\al, \be, \la, \mu \in \cP$,
										\begin{align*}
											\sum_{i,j} (-1)^{|\al|-|\al\uparrow j|+i+j} \left( s^\perp_{\la\downarrow i}(\bfp) s^\perp_\mu(\bfq) \tau_m(\bfp,\bfq) \right)
																									\left( s^\perp_{\al\uparrow j}(\bfp) s^\perp_\be(\bfq) \tau_{k+1}(\bfp,\bfq) \right) \\
													= \sum_{s,t} (-1)^{|\mu|-|\mu\uparrow s|+s+t} \left( s^\perp_\la(\bfp) s^\perp_{\mu\uparrow s}(\bfq) \tau_{m+1}(\bfp,\bfq) \right)
																										\left( s^\perp_\al(\bfp) s^\perp_{\be\downarrow t}(\bfq) \tau_k(\bfp,\bfq) \right),
										\end{align*}
									where the first sum is over integers $i,j \geq 1$ such that $|\la\downarrow i| + |\al\uparrow j| = |\la| + |\al| + m - k$ and the second sum
									is over integers $s,t \geq 1$ such that $|\mu\uparrow s| + |\be\downarrow t| = |\mu| + |\be| + k - m$.
	\end{enumerate}
\end{thm}
\begin{proof}
	The proof is essentially the same as Theorem~\ref{kphierarchy}.
\end{proof}

\begin{exa}
	In the case of the 2-Toda hierarchy, choose $k = m-1, \al = \la = \be = \epsilon$ and $\mu = 1$. We compute $\epsilon \uparrow 1 = \epsilon$,
	$\epsilon \uparrow 2 = 1$, $\epsilon \downarrow 1 = \epsilon$, $\epsilon \downarrow 2 = 1$, $1 \uparrow 1 = \epsilon$ and $1 \uparrow 2 = 1^2$.
	The only solutions $(i,j)$ to $|\la \downarrow i| + |\al \uparrow j| = |\la| + |\al| + 1$ are $(1,2)$ and $(2,1)$. Similarly, the only solution
	$(s,t)$ to $|\mu \uparrow s| + |\be \downarrow t| = |\mu| + |\be| - 1$ is $(1,1)$. Theorem~\ref{2todahierarchy}(iii) then gives
		\[ a_1^\ep(m)a_\ep^1(m) - a_1^1(m) a_\ep^\ep(m) = -a_\ep^\ep(m+1) a_\ep^\ep(m-1) \]
	as one of the coefficient constraints. Similarly, Theorem~\ref{2todahierarchy}(iv) gives
		\begin{align*}
			\left( s_\ep^\perp(\bfp) s_\ep^\perp(\bfq) \tau_{m+1} \right) \left( s_\ep^\perp(\bfp) s_\ep^\perp(\bfq) \tau_{m-1} \right)
				&+ \left( s_1^\perp(\bfp) s_\ep^\perp(\bfq) \tau_m \right) \left( s_\ep^\perp(\bfp) s_1^\perp(\bfq) \tau_m \right) \\
				&= \left( s_1^\perp(\bfp) s_1^\perp(\bfq) \tau_m \right) \left( s_\ep^\perp(\bfp) s_\ep^\perp(\bfq) \tau_m \right)
		\end{align*}
	as one of the partial differential equations in the 2-Toda hierarchy. Using the fact that $s_\ep^\perp(\bfp) = 1$ and
	$s_1^\perp(\bfp) = \frac{\del}{\del p_1}$, this gives
		\[ \tau_{m+1} \tau_{m-1} + \left( \frac{\del}{\del p_1} \tau_m \right) \left( \frac{\del}{\del q_1} \tau_m \right)
					= \tau_m \left( \frac{\del^2}{\del p_1 \del q_1} \tau_m \right), \]
	which can be further simplified to 
		\[ \frac{\del^2}{\del p_1 \del q_1} \log \tau_m = \frac{\tau_{m+1}\tau_{m-1}}{\tau_m^2}. \]
	This last equation is called the 2-Toda equation.
\end{exa}

We will now take this opportunity to discuss some of the integrable hierarchies that reside within of the 2-Toda hierarchy.

\begin{thm}\label{subhie}
	Suppose the sequence of formal power series $\{\tau_m(\bfp,\bfq)\}_{m \in \bbZ}$ is a solution to the 2-Toda hierarchy. Then for any $m \in \bbZ$, $r > 0$
	and $\la, \al \in \cP$, we have
	\begin{equation*}
		\sum_{i,j} (-1)^{|\al| + |\al\uparrow j| + i + j} \left( s^\perp_{\la\downarrow i}(\bfp) \tau_m(\bfp,\bfq) \right)
																											\left( s^\perp_{\al\uparrow j}(\bfp) \tau_{m-r+1}(\bfp,\bfq) \right) = 0,
	\end{equation*}
	where the sum is over $i,j \geq 1$ such that $|\la\downarrow i| + |\al\uparrow j| = |\la| + |\mu| + r$.	By symmetry, this also implies that
	\begin{equation*}
		\sum_{i,j} (-1)^{|\al| + |\al\uparrow j| + i + j} \left( s^\perp_{\la\downarrow i}(\bfq) \tau_m(\bfp,\bfq) \right)
																											\left( s^\perp_{\al\uparrow j}(\bfq) \tau_{m-r+1}(\bfp,\bfq) \right) = 0,
	\end{equation*}
	where the sum is over $i,j \geq 1$ such that $|\la\downarrow i| + |\al\uparrow j| = |\la| + |\mu| + r$.
\end{thm}

\begin{proof}
	If we set $k = m - r$ and $\mu = \be = \epsilon$ in Theorem~\ref{2todahierarchy}(iv) then the right hand side of the equation is a sum over integers
	$s, t \geq 1$ such that
	\begin{equation*}
		|\mu \uparrow s| + |\be \downarrow t| = |\mu| + |\be| + k - m = -r < 0,
	\end{equation*}
	and so there can be no solutions $s,t \geq 1$. The result then follows.
\end{proof}

In particular, if we set $r = 1$ in Theorem~\ref{subhie} then the resulting family of partial differential equations are those found in Theorem~\ref{kphierarchy}(iv).
Theorem~\ref{subhie} thus implies the well known result that if $\{\tau_m(\bfp,\bfq)\}_{m \in \bbZ}$ is a solution to the 2-Toda hierarchy, then each
$\tau_m$ is a solution to the KP hierarchy in $\bfp$ and $\bfq$ independently.

\section{Diagonal Solutions}

For the remainder of this paper we will assume that the sequence of formal power series $\{ \tau_n(\bfp,\bfq) \}_{n \in \bbZ}$ is diagonal. That is, for
each $n \in \bbZ$,
	\[ \tau_n(\bfp, \bfq) = \sum_{\la \in \cP} g_\la(n) s_\la(\bfp) s_\la(\bfq), \]
where the $g_\la(n), \la \in \cP, n \in \bbZ$ are scalars.

We begin by describing a specialization to diagonal solutions of the characterization of solutions to the 2-Toda hierarchy

\begin{thm}\label{diagonalsolutions}
	The sequence of formal power series $\{ \tau_n(\bfp, \bfq) \}_{n \in \bbZ}$ is a solution to the 2-Toda hierarchy if and only if for all
	$\la, \mu \in \cP$, $n, m \in \bbZ$ and integers $i,j \geq 1$ such that $|\la| + |\mu| = |\la \uparrow i| + |\mu \downarrow j| + n - m - 1$, we have
		\[ g_{\la \uparrow i}(n) g_{\mu \downarrow j}(m) = g_\la(n-1) g_\mu(m+1). \]
\end{thm}
\begin{proof}
	We begin with the characterization of the 2-Toda hierarchy given in Theorem~\ref{2todahierarchy}(iii). Recall that $\{\tau_m(\bfp,\bfq)\}_{m \in \bbZ}$
	is a solution to the 2-Toda hierarchy if and only if for all $m, n \in \bbZ$,  $\la, \mu, \al, \be \in \cP$ we have
		\begin{align}\label{charloc}
			\sum_{i,j} (-1)^{|\al|-|\al\uparrow j|+i+j} & a_\mu^{\la\downarrow i}(m) a_\be^{\al\uparrow j}(n+1) \\
					& = \sum_{s,t} (-1)^{|\mu| + |\mu\uparrow s|+s+t} a^\la_{\mu\uparrow s}(m+1) a^\al_{\be\downarrow t}(n), \notag
		\end{align}
	where the first sum is over $i, j \geq 1$ such that $|\la\downarrow i| + |\al\uparrow j| = |\la| + |\al| + m - n$ and the second sum
	is over $s, t \geq 1$ such that $|\mu\uparrow s| + |\be\downarrow t| = |\mu| + |\be| + n - m$.

	In order to have a non-trivial sum on the left hand side, it must be true that $\mu = \la \downarrow i$, $\be = \al \uparrow j$ for
	some suitable $\la, \al, i$ and $j$. However, using the fact that
	\[ (\al \uparrow j) \downarrow(u_j(\al) + 1) = \al, \]
	and
	\[ (\la \downarrow i) \uparrow (\la_i + 1) = \la, \]
	we see that $\la = \mu \uparrow(\la_i + 1)$ and $\al = \be \downarrow(u_j(\al) + 1)$ and so \eqref{charloc} becomes
	\[ (-1)^{|\al| + |\al\uparrow j| + i + j} g_{\la \downarrow i}(m) g_{\al \uparrow j}(n+1)
			= (-1)^{|\la \downarrow i| - |\la| + \la_i + u_j(\al)} g_\la(m+1) g_\al(n). \]
	Using the fact that
	\[ |\al \uparrow j| = |\al| + j + u_j(\al) - 1, \]
	we have $|\al| - |\al \uparrow j| + i + j = u_j(\al) + i + 1$ and using the fact that
	\[ |\la \downarrow i| = |\la| + i - \la_i - 1, \]
	we have $|\la \downarrow i| - |\la| + \la_i + u_j(\al) = u_j(\al) + i - 1$.
	
	Thus, we get
	\[ g_{\la\downarrow i}(m)g_{\al\uparrow j}(n+1) = g_\la(m+1) g_\al(n) \]
	where $i, j \geq 1$ are such that
	\[ |\la \downarrow i| + |\al \uparrow j| = |\la| + |\al| + m - n. \]
	
	After shifting $n \mapsto n - 1$ and reindexing, the result follows.
\end{proof}

We now describe the content-type series. We begin by defining the shifted content products and discussing some of their properties.

For $m, k \in \bbZ$ we define
	\[ Y(m, k) := \begin{cases}
								\prod_{j=1}^k y_{m+1-j}, & \mbox{ if } k \geq 1, \\
								1,											 & \mbox{ if } k = 0, \\
								Y(m-k, -k)^{-1},         & \mbox{ if } k \leq -1.
								\end{cases} \]
Also, for any $\la \in \cP$, $n \in \bbZ$,
	\[ Y_n(\la) := \prod_{i=1}^{\ell(\la)} Y(\la_i - i + n, \la_i). \]
Note that $Y_n(\la)$ is the shifted content product in the indeterminates $y_i$ for the partition $\la$, i.e.,
	\[ Y_n(\la) = \prod_{\square \in \la} y_{n+c(\square)} \]
where $c(\square)$ is the content of the cell $\square \in \la$.

\begin{lem}\label{contentlemmas}
	Suppose $j, k, s \in \bbZ$.
	\begin{enumerate}
		\item[(i)]	\[ \frac{Y(s,k)}{Y(s,j)} = \frac{1}{Y(s-k,j-k)} = Y(s-j,k-j), \]
		\item[(ii)] \[ \frac{Y(j,j)}{Y(k,k)} = Y(j,j-k), \]
		\item[(iii)] and, if $j, k \geq 0$,
								 \[ \frac{Y(s+j-k,j)}{Y(s,k)} = Y(s+j-k,j-k). \]
	\end{enumerate}
\end{lem}
\begin{proof}
	Each of the identities follows in a straightforward way by first separating into cases and then applying the definition.
\end{proof}

For any $n \in \bbZ$, define
	\[ \ta_n := \begin{cases}
								ab^ny_0^{n/2} \prod_{i=1}^{n-1}Y(i,i), & \mbox{ if } n > 0, \\
								a,                                     & \mbox{ if } n = 0, \\
								ab^ny_0^{n/2} \prod_{i=1}^{-n-1}Y(-i-1,-i-1)^{-1}, & \mbox{ if } n < 0.
							\end{cases} \]
We now define a sequence of formal power series which we call content-type series. For $n \in \bbZ$, define
	\[ \Phi_n(\bfp, \bfq; a, b, \bfy) = \sum_{\la \in \cP} \ta_n Y_n(\la) s_\la(\bfp)s_\la(\bfq). \]
Our aim now is to show that the sequence of formal power series $\{\Phi_n\}_{n\in\bbZ}$ is a solution to the 2-Toda hierarchy. Note that this
result was originally proven by Orlov and Shcherbin in \cite{OS1}, however, their approach is different from ours.

\begin{lem}\label{constantratio}
	For any integer $m$ we have
		\[ \frac{\ta_{m+1}}{\ta_m} = by_0^{1/2}Y(m,m). \]
\end{lem}
\begin{proof}
	If $m \geq 0$ then
		\begin{align*}
			\frac{\ta_{m+1}}{\ta_m} &= \frac{ab^{m+1}y_0^{(m+1)/2}\prod_{i=1}^{m}Y(i,i)}{ab^{m}y_0^{m/2}\prod_{i=1}^{m-1}Y(i,i)}, \\
													&= by_0^{1/2}Y(m,m).
		\end{align*}
	If $m < 0$ we have
		\begin{align*}
			\frac{\ta_{m+1}}{\ta_m} &= \frac{ab^{m+1}y_0^{(m+1)/2}\prod_{i=1}^{-m-2} Y(-i-1,-i-1)^{-1}}{ab^{m}y_0^{m/2}\prod_{i=1}^{-m-1} Y(-i-1,-i-1)^{-1}}, \\
													&= by_0^{1/2} Y(m,m).
		\end{align*}
\end{proof}

\begin{lem}\label{uparrowratio}
	For integer $n$, integer $i > 0$ and partition $\la$,
		\[ \frac{Y_n(\la\uparrow i)}{Y_{n-1}(\la)} = Y(|\la\uparrow i| - |\la| + n-1, |\la\uparrow i| - |\la|). \]
\end{lem}
\begin{proof}
	We have
		\begin{align*}
			\frac{Y_n(\la\uparrow i)}{Y_{n-1}(\la)}
					&= Y(i-1-(u_i(\la)+1)+n)
									\left[ \frac{ \prod_{k=1}^{u_i(\la)} Y(\la_k - 1 - k + n, \la_k - 1) }{ \prod_{k=1}^{u_i(\la)} Y(\la_k - k + n - 1, \la_k)} \right] \\
					&\qquad \qquad \times \left[ \frac{ \prod_{k=u_i(\la)+1}^{\ell(\la)} Y(\la_k - (k+1) + n, \la_k)}{ \prod_{k=u_i(\la)+1}^{\ell(\la)} Y(\la_k-k+n-1,\la_k)} \right], \\
					&= \frac{Y(i-u_i(\la)-1+n-1,i-1)}{\prod_{k=1}^{u_i(\la)} Y(n-k,1)},
		\end{align*}
	where the denominator comes from an application of Lemma~\ref{contentlemmas}(i). Simplifying and applying Lemma~\ref{contentlemmas}(iii) gives
		\begin{align*}
			\frac{Y_n(\la\uparrow i)}{Y_{n-1}(\la)}
					&= \frac{Y(i-u_i(\la)-1+n-1,i-1)}{Y(n-1,u_i(\la))}, \\
					&= Y(i-u_i(\la)-1+n-1,i-u_i(\la)-1), \\
					&= Y(|\la\uparrow i|-|\la| + n-1, |\la\uparrow i|-|\la|).
		\end{align*}
\end{proof}

\begin{lem}\label{downarrowratio}
	For integer $m$, integer $j > 0$ and partition $\mu$, we have
		\[ \frac{Y_{m+1}(\mu)}{Y_m(\mu\downarrow j)} = Y(|\mu| - |\mu\downarrow j| + m, |\mu| - |\mu\downarrow j|). \]
\end{lem}
\begin{proof}
	We have
		\begin{align*}
			\frac{Y_{m+1}(\mu)}{Y_m(\mu\downarrow j)}
					&= Y(\mu_j - j + m + 1, \mu_j)
								\left[ \frac{\prod_{k=1}^{j-1} Y(\mu_k-k+m+1,\mu_k)}{\prod_{k=1}^{j-1}Y(\mu_k-k+m+1,\mu_k+1)} \right] \\
					&\qquad \qquad \times \left[ \frac{\prod_{k=j+1}^{\ell(\mu)} Y(\mu_k-k+m+1,\mu_k)}{\prod_{k=j+1}^{\ell(\mu)}Y(\mu_k-(k-1)+m,\mu_k)} \right], \\
					&= \frac{Y(\mu_j-j+m+1,\mu_j)}{\prod_{k=1}^{j-1}Y(m-k+1,1)},
		\end{align*}
	where the denominator comes from an application of Lemma~\ref{contentlemmas}(i). After simplifying and applying Lemma~\ref{contentlemmas}(iii) we have
		\begin{align*}
			\frac{Y_{m+1}(\mu)}{Y_m(\mu\downarrow j)}
					&= \frac{Y(\mu_j-j+m+1,\mu_j)}{Y(m,j-1)}, \\
					&= Y(\mu_j-j+m+1,\mu_j-j+1), \\
					&= Y(|\mu|-|\mu\downarrow j|+m, |\mu|-|\mu\downarrow j|).
		\end{align*}
\end{proof}

\begin{thm}\label{contentsolution}
	The sequence of formal power series $\{\Phi_n(\bfp,\bfq;a,b,\bfy)\}_{n\in\bbZ}$ is a solution to the 2-Toda hierarchy.
\end{thm}
\begin{proof}
	By Theorem~\ref{diagonalsolutions} we need to show that for all partitions $\la, \mu$ and integers $i, j, n, m$ such that $i, j \geq 1$ and
		\[ |\la| + |\mu| = |\la\uparrow i| + |\mu\downarrow j| + n -m - 1, \]
	we have
		\[ g_{\la\uparrow i}(n)g_{\mu\downarrow j}(m) = g_\la(n-1)g_\mu(m+1). \]
	In the case of the content-type series, this becomes
		\[ \ta_n Y_n(\la\uparrow i) \ta_m Y_m(\mu\downarrow j) = \ta_{n-1} Y_{n-1}(\la) \ta_{m+1} Y_{m+1}(\mu), \]
	or, after rearranging,
		\begin{equation}\label{goaleqn}
			\frac{ Y_n(\la\uparrow i) / Y_{n-1}(\la) }{ Y_{m+1}(\mu) / Y_m(\mu\downarrow j) } = \frac{ \ta_{m+1} / \ta_m }{ \ta_n / \ta_{n-1} }.
		\end{equation}
	So, if we can show that the identity \eqref{goaleqn} holds then we are done.
	
	Using Lemma~\ref{uparrowratio} and Lemma~\ref{downarrowratio}, the left hand side of \eqref{goaleqn} becomes
		\[ \frac{Y(|\la\uparrow i| - |\la| + n - 1, |\la\uparrow i| - |\la|)}{Y(|\mu|-|\mu\downarrow j| + m, |\mu| - |\mu\downarrow j|)}. \]
	Using the fact that $|\la| + |\mu| = |\la\uparrow i| + |\mu\downarrow j| + n - m - 1$, the left hand side of \eqref{goaleqn} then becomes
		\[ \frac{Y(|\la\uparrow i|-|\la|+n - 1,|\la\uparrow i|-|\la|)}{Y(|\la\uparrow i|-|\la|+n - 1,|\la\uparrow i|-|\la| -m - 1 + n)}. \]
	Applying Lemma~\ref{contentlemmas}(i) then tells us that the left hand side of \eqref{goaleqn} is
		\begin{equation}\label{missionaccomplished} Y(m,m+1-n). \end{equation}
	
	Now, using Lemma~\ref{constantratio}, the right hand side of \eqref{goaleqn} becomes
		\[ \frac{Y(m,m)}{Y(n-1,n-1)}. \]
	Applying Lemma~\ref{contentlemmas}(ii) then gives \eqref{missionaccomplished}. Hence \eqref{goaleqn} is satisfied and the content-type series solves
	the 2-Toda hierarchy.
\end{proof}

We now give a partial converse to this theorem, showing that many diagonal series that arise as solutions to the 2-Toda hierarchy are in fact content-type series.

\begin{thm}
	For $n \in \bbZ$ suppose that
		\[ \tau_n = \sum_{\la \in \cP} g_\la(n) s_\la(\bfp) s_\la(\bfq), \]
	where $g_\la(n), \la \in \cP, n \in \bbZ$ are scalars and that $\{\tau_n(\bfp,\bfq)\}_{n\in\bbZ}$ is a solution to the 2-Toda hierarchy. Define a sequence
	$\bfu = (\cdots, u_{-1}, u_0, u_1, \cdots)$ as follows. For $m \geq 0$,
	\begin{align*}
		g_{m+1}(0) &= u_m g_m(0), \\
		g_{1^{m+1}}(0) &= u_{-m} g_{1^m}(0).
	\end{align*}
	If the sequences $\{g_m(0)\}_{m\geq 0}$ and $\{g_{1^m}(0)\}_{m\geq 0}$ can be uniquely recovered from $g_\ep(0)$ and the sequence $\bfu$ and
	if $g_\ep(0), g_1(0) \not = 0$ then for all $n \in \bbZ$,
		\[ \tau_n = \Phi_n\left(\bfp,\bfq; g_\ep(0), \frac{g_\ep(1)}{u_0^{1/2} g_\ep(0)}, \bfu\right). \]
\end{thm}

Note that the condition on the sequence $\bfu$ is essentially that if $g_m(0) = 0$ then $g_M(0) = 0$ for all $M \geq m$ and similarly for $g_{1^m}(0)$. Also
note that the choice of 0 in the above statement is not special since the conditions on the coefficients of solutions to the 2-Toda hierarchy are
invariant under the translation $\tau_n \mapsto \tau_{n+1}$.

\begin{proof}
	The approach taken will be to show that if the conditions in the theorem are satisfied then each of the coefficients $g_\la(n)$ can be constructed from the
	sequence $\bfu$ along with $g_\ep(0)$ and $g_\ep(1)$. The result then follows from the fact that $\Phi_n$ also satisfies the conditions in the theorem
	and that if $\tau_n = \Phi_n$ then for $m \geq 0$,
		\[ g_{m+1}(0) = \ta_0 Y_0(m+1) = y_m \ta_0 Y_0(m) = y_m g_m(0), \]
		\[ g_{1^{m+1}}(0) = \ta_0 Y_0(1^{m+1}) = y_{-m} \ta_0 Y_0(1^m) = y_{-m} g_{1^m}(0), \]
	and that $g_\ep(0) = \ta_0 Y_0(\ep) = a$, and $g_\ep(1) = \ta_1 Y_1(\ep) = a b y_0^{1/2}.$
	
	In the following we will use the notation $(r, \eta)$ where $\eta \in \cP$ and $r \geq \eta_1$ to denote the partition $(r, \eta_1, \eta_2, \cdots)$ and
	we will use the notation $1^s + \eta$ where $\eta \in \cP$ and $s \geq \ell(\eta)$ to denote the partition $(\eta_1 + 1, \cdots, \eta_s + 1)$ with
	the convention that $\eta_i = 0$ if $i > \ell(\eta)$.
	
	Since $\{\tau_n\}_{n\in\bbZ}$ satisfies the 2-Toda hierarchy and is a diagonal solution, we know that for any partitions $\la, \mu \in \cP$, $m, n \in \bbZ$ and
	integers $i,j \geq 1$ such that
		\[ |\la| + |\mu| = |\la \uparrow i| + |\mu \downarrow j| + n - m - 1, \]
	the identity
		\begin{equation}\label{identityA}
			g_{\la\uparrow i}(n) g_{\mu \downarrow j}(m) = g_\la(n-1) g_\mu(m+1)
		\end{equation}
	holds.
	
	If we choose $n = m = 0, i = 2, j = 1, \la = \mu = \ep$ in \eqref{identityA} then we get
		\[ g_1(0) g_\ep(0) = g_\ep(-1) g_\ep(1). \]
	Since $g_1(0), g_\ep(0) \not = 0$ we know that $g_\ep(1), g_\ep(-1) \not = 0$.
	
	If we choose $n = 0, \la = \ep$ in \eqref{identityA} then we get
		\[ g_k(0) g_{\mu\downarrow j}(m) = g_\ep(-1)g_\mu(m+1), \]
	where
		\[ k = |\mu| - |\mu \downarrow j| + m + 1. \]
	In particular, if $\eta \in \cP$, $r \geq \eta_1$ and $m \geq -1$, this implies that
		\begin{equation}\label{recursion1}
			g_\ep(-1)g_{(r,\eta)}(m+1) = g_{r+m+1}(0) g_\eta(m).
		\end{equation}
	Since $g_\ep(-1) \not = 0$, equation \eqref{recursion1} allows us to construct $g_\eta(n), n > 0$ inductively provided we know
	$g_\la(0)$ for any partition $\la$.
	
	Similarly, if we choose $m = 0, \mu = \ep$ in \eqref{identityA} then we get
		\[ g_{\la \uparrow i}(n) g_{1^k}(0) = g_\la(n-1)g_\ep(1), \]
	where
		\[ k = |\la| - |\la \uparrow i| - n + 1. \]
	In particular, if $\eta \in \cP$, $s \geq \ell(\eta)$ and $n \leq 1$ then
		\begin{equation}\label{recursion2}
			g_\ep(1)g_{1^s + \eta}(n-1) = g_{1^{s-n+1}}(0) g_\eta(n).
		\end{equation}
	Since $g_\ep(1) \not = 0$, equation \eqref{recursion2} allows us to construct $g_\eta(n), n < 0$ inductively provided we know
	$g_\la(0)$ for any partition $\la$.
	
	All that is left now is to show that $g_\la(0)$ can be constructed provided we know $g_m(0)$, $g_{1^m}(0)$, $g_\ep(0)$, $g_\ep(1)$ and $g_\ep(-1)$.
	
	If we choose $m = -1$ in \eqref{recursion1} then we have, for any $\eta \in \cP$ and $r \geq \eta_1$,
		\begin{equation}\label{recursion3}
			g_\ep(-1) g_{(r,\eta)}(0) = g_r(0) g_\eta(-1).
		\end{equation}
	If we choose $n = 0$ in \eqref{recursion2} then we have, for any $\eta \in \cP$ and $s \geq \ell(\eta)$,
		\begin{equation}\label{recursion4}
			g_\ep(1) g_{1^s + \eta}(-1) = g_{1^{s+1}}(0) g_\eta(0).
		\end{equation}
	Note that on the right hand side of both \eqref{recursion3} and \eqref{recursion4} the partition corresponding to the unknown coefficient is
	strictly smaller in size than the partition corresponding to the unknown coefficient on the left hand side. Thus, by induction, we can
	construct $g_\la(0)$ for every partition $\la$.
	
	To see how the last part about $g_\la(0)$ works, suppose we wish to determine $g_{42^21^3}(0)$. By repeatedly applying \eqref{recursion3} and
	\eqref{recursion4} we get
		\[ g_\ep(-1) g_{42^21^3}(0) = g_4(0) g_{2^21^3}(-1), \]
	and
		\[ g_\ep(1) g_{2^21^3}(-1) = g_{1^6}(0)g_{1^2}(0). \]
\end{proof}

Note that not all diagonal solutions to the 2-Toda hierarchy are content-type series. For example, it is not difficult to show that
if we define $\tau_0 = s_1(\bfp) s_1(\bfq) = p_1 q_1$ and $\tau_n = 0$ for $n \not = 0$ then $\{\tau_n\}_{n \in \bbZ}$ is a
solution to the 2-Toda hierarchy.

\section{Examples}

In this section we briefly discuss a few content-type series that have arisen in the literature. We begin with a family of series that was
presented in \cite{GJ1} and which encodes a number of enumerative generating functions.

For $a_1, a_2, \cdots \geq 0$ and $\al, \be \in \cP$ with $|\al| + |\be| = d$, let $\cB_{\al,\be}^{a_1, a_2, \cdots}$ be the set of tuples
of permutations $(\si, \ga, \pi_1, \pi_2, \cdots)$ on $\{1, \cdots, d\}$ such that
	\begin{enumerate}
		\item[(i)]	The permutation $\si$ has cycle type $\al$, $\ga$ has cycle type $\be$ and $d - \ell(\pi_i) = a_i$ for $i \geq 1$ where
								$\ell(\pi_i)$ is the number of cycles in the disjoint cycle decomposition of $\pi_i$;
		\item[(ii)]	$\si \ga \pi_1 \pi_2 \cdots = $ id.
	\end{enumerate}
Let $b_{\al,\be}^{a_1, a_2, \cdots}$ be the number of tuples in $\cB_{\al,\be}^{a_1, a_2, \cdots}$ (Note that our $b_{\al,\be}^{a_1,a_2,\cdots}$
is written $\tilde{b}_{\al,\be}^{a_1,a_2,\cdots}$ in \cite{GJ1}). The tuples counted by $b_{\al,\be}^{a_1,a_2,\cdots}$ are called
constellations in \cite{LZ1}.

If we now construct the generating function for these numbers,
	\[ B := \sum_{\substack{\al,\be \in \cP, \\ |\al| = |\be| = d \geq 1, \\ a_1, a_2, \cdots \geq 0}}
					\frac{1}{d!} b_{\al,\be}^{a_1, a_2, \cdots} p_\al p_\be u_1^{a_1} u_2^{a_2} \cdots, \]
then using representation theory (see \cite{GJ1} for details) it can be shown that
	\[ B = \left. \Phi_0 \right|_{a = b = 1, \,\, y_j = \prod_{i \geq 1} (1 + ju_i), \,\, j \in \bbZ}, \]
and hence via Theorem~\ref{contentsolution} can be embedded in a solution to the 2-Toda hierarchy.

Various specializations of $B$ give rise to generating functions for a number of different enumerative problems such as map and hypermap
enumeration. We choose to focus on only a single specialization here, namely to the double Hurwitz problem, and defer to \cite{GJ1}
for others.

Double Hurwitz numbers arise in the enumeration of branched covers of the sphere through an encoding due to Hurwitz \cite{Hur1}.
Using the infinite wedge space formalism, Okounkov \cite{Ok1} showed that the generating function for double Hurwitz numbers can be embedded
in a solution to the 2-Toda hierarchy. We will see that this also follows from the fact that the generating function for double Hurwitz
numbers is a specialization of the generating series $B$ described above. Note that Orlov \cite{Orlov1} has also given a proof of this
result that is similar to ours although from a different point of view.

For $\al, \be \in \cP$, $|\al| = |\be| = d$ and $g \geq 0$ let $r_{\al,\be}^g = \ell(\al) + \ell(\be) + 2g - 2$. The Hurwitz number $H_{\al,\be}^g$
is defined by
	\[ H_{\al,\be}^g := \frac{1}{d!} |\mbox{Aut} \al| |\mbox{Aut} \be| b_{\al,\be}^{a_1, a_2, \cdots}, \]
where $a_i = 1$ for $1 \leq i \leq r_{\al,\be}^g$ and $a_i = 0$ for $i > r_{\al,\be}^g$. The (disconnected) double Hurwitz series is then
	\[ H := \sum_{\substack{\al,\be \in \cP \\ |\al| = |\be| = d \geq 1, \\ g \ge 0}} \frac{ H_{\al,\be}^g }{ |\mbox{Aut} \al| |\mbox{Aut} \be| }
																																		p_\al q_\be \frac{t^{r_{\al,\be}^g}}{r_{\al,\be}^g !}. \]
After doing some algebraic manipulations, it can be shown (see \cite{GJ1} for details) that
	\[ H = \left. B \right|_{a = b = 1, \,\, e_k(u_1,u_2,\cdots) = \frac{t^k}{k!} \,\, k \geq 1} = \left. \Phi_0 \right|_{a = b = 1, \,\, y_j = e^{jt}, \,\, j \in \bbZ}, \]
where the $e_k$ are the elementary symmetric functions.
Thus, we see that the double Hurwitz series can be embedded in a solution to the 2-Toda hierarchy. This fact was used in \cite{Ok1} to prove a
conjecture of Pandharipande\cite{Pand1} concerning the simple Hurwitz numbers.



We now give an example of a content-type series arising in the study of random partitions. We begin by defining a function $\cM$ on the set of partitions.
For $\la \in \cP$,
	\[ \cM(\la) = \frac{1}{Z} s_\la(\bfp) s_\la(\bfq) \]
where $Z = \sum_{\la \in \cP} s_\la(\bfp) s_\la(\bfq)$. The function $\cM$, considered as a probability measure on partitions, is called the Schur measure
and was introduced by Okounkov in \cite{Ok2}. The Schur measure can be thought of as a generalization of the $z$-measure introduced by Borodin and Olshanski
(see \cite{BO1} for example) and as such contains as specializations or limiting cases a variety of other probability models\cite{Bor1}. One of the results
in \cite{Ok2} is that the correlation functions of the Schur measure satisfy the 2-Toda hierarchy in the parameters of the measure. Here we show that this
also follows from Theorem~\ref{diagonalsolutions}.

For any partition $\la$, let $\fS(\la) = \{\la_i - i\}_{i \geq 1}$. This can be thought of as the sequence of contents of the rightmost cell in each
row of $\la$ where we append a countable number of parts of size 0 to $\la$. Note that $\la$ is uniquely recoverable from $\fS(\la)$.
For any $X \subseteq \bbZ$, we define
	\[ \rho(X) = \sum_{\substack{\la \in \cP \\ X \subseteq \fS(\la)}} \cM(\la). \]
The $\rho(X)$ are the correlators of the Schur measure. If $X = \{x_1, x_2, \cdots\}$ then for $n \in \bbZ$ we use $X - n$ to denote the set
$\{x_1 - n, x_2 - n, \cdots\}$.

\begin{pro}
	For any $X \subseteq \bbZ$, the sequence $\{Z \rho(X - n) \}_{n \in \bbZ}$ is a solution to the 2-Toda hierarchy.
\end{pro}

\begin{proof}
	First, notice that
		\[ \tau_n = Z \rho(X - n) = \sum_{\substack{\la \in \cP \\ X - n \subseteq \fS(\la)}} s_\la(\bfp) s_\la(\bfq). \]
	
	Suppose $\la, \mu \in \cP$ and $n, m \in \bbZ$ are such that
		\[ X - n + 1 \subseteq \fS(\la), \]
	and
		\[ X - m - 1 \subseteq \fS(\mu). \]
	Further, suppose that $i, j \in \bbZ$ are such that $|\la| + |\mu| = |\la \uparrow i| + |\mu \downarrow j| + n - m - 1$. We will now show
	that
		\[ X - n \subseteq \fS(\la \uparrow i) \]
	and
		\[ X - m \subseteq \fS(\mu \downarrow j) \]
	which, by Theorem~\ref{diagonalsolutions}, will prove the proposition.
	
	First, recall that $\la \uparrow i = (\la_1 - 1, \cdots, \la_{u_i(\la)} - 1, i - 1, \la_{u_i(\la) + 1}, \cdots)$ where $u_i(\la)$ is the
	unique integer such that $\la_{u_i(\la)} \geq i > \la_{u_i(\la) + 1}$ and that $\mu \downarrow j = (\mu_1 + 1, \cdots, \mu_{j-1} + 1, \mu_{j+1}, \cdots)$.
	Thus, we have $|\la| - |\la \uparrow i| = u_i(\la) + 1 - i$ and $|\mu| - |\mu \downarrow j| = \mu_j - j + 1$ so that the constraint
	$|\la| + |\mu| = |\la \uparrow i| + |\mu \downarrow j| + n - m - 1$ becomes $u_i(\la) - i + \mu_j - j + 2 = n - m - 1$.
	
	Now, for any $x \in X$, we know that
		\[ x - n + 1 = \la_t - t \]
	for some positive integer $t$ (recall that $\la_t = 0$ if $t > \ell(\la)$). If $t \leq u_i(\la)$ then
		\[ x - n = \la_t - 1 - t = (\la \uparrow i)_t - t. \]
	Similarly, if $t > u_i(\la)$ then
		\[ x - n = \la_t - (t + 1) = (\la \uparrow i)_{t+1} - (t + 1). \]
	In particular, notice that this implies there is no $x \in X$ such that
		\[ x - n = (\la \uparrow i)_{u_i(\la) + 1} - (u_i(\la) + 1) = i - u_i(\la) - 2. \]
	
	So far we have shown that $X - n \subseteq \fS(\la \uparrow i)$ and so now we must show that $X - m \subseteq \fS(\mu \downarrow j)$.
	
	For any $x \in X$, we know that
		\[ x - m - 1 = \mu_t - t \]
	for some positive integer $t$. If $t < j$ then
		\[ x - m = (\mu_t + 1) - t = (\mu \downarrow j)_t - t. \]
	Similarly, if $t > j$ then
		\[ x - m = \mu_t - (t - 1) = (\mu \downarrow j)_{t-1} - (t-1). \]
	If $t = j$ then since $x - m - 1 = \mu_j - j$ and $u_i(\la) - i + \mu_j - j + 2 = n - m - 1$ we have
		\[ x - n = m + 1 - n + \mu_j - j = i - u_i(\la) - 2, \]
	a contradiction and hence $X - m \subseteq \fS(\mu \downarrow j)$.
\end{proof}

For the last example we look at a matrix integral that arises in mathematical physics and, more recently, in a combinatorial context\cite{GGN1,GGN2}.
Consider the integral
	\[ I_n := \int_{U(n)} e^{\mbox{Tr}(XUYU^*)}dU, \]
where $U(n)$ is the group of $n$ by $n$ unitary matrices, $dU$ is the Haar measure on $U(n)$ and $X$ and $Y$ are diagonal matrices.

The integral $I_n$ appeared recently in the context of a combinatorial problem involving enumeration of certan restricted factorizations in the symmetric group.
This problem is called the monotone double Hurwitz problem and is related to the double Hurwitz problem discussed earlier. We refer to
\cite{GGN1} for an analytic treatment of $I_n$ and \cite{GGN2} for a solution to the corresponding enumeration problem.

If we let $p_k = \mbox{Tr}(X^k)$ and $q_k = \mbox{Tr}(Y^k)$ then using the character expansion method \cite{Orlov1} we have
	\[ I_n = \sum_{\substack{\la \in \cP, \\ \ell(\la) \leq n}} \frac{s_\la(\bfp) s_\la(\bfq)}{\prod_{\square \in \la} (n + c(\square))}. \]
Given a partition $\la$, the cell with the smallest content is in the first column and the last row and the content of this cell is $1 - \ell(\la)$.
From this it is easy to see that $\la \in \cP$ is such that $\ell(\la) \leq n$ if and only if for all $\square \in \la$, $c(\square) > -n$. Now, consider
the series $\tilde{I_n} = \Phi_n$ where we set $a = 1, b = y_0^{-1/2}$ and $y_i = \frac{1}{i}$ if $i > 0$ and $y_i = 0$ if $i \leq 0$.
Note that we first make the substitution $b = y_0^{-1/2}$ and each of the coefficients in the resulting series is a monomial in the $y_i$ and so we may set $y_0 = 0$.
We have
	\[ \tilde{I_n} = \tilde{\ta_n} \sum_{\la \in \cP} \tilde{Y_n}(\la) s_\la(\bfp) s_\la(\bfq), \]
where
	\[ \tilde{\ta_n} = \left( \prod_{i = 1}^{n-1} i! \right)^{-1}. \]
Also,
	\[ \tilde{Y_n}(\la) = \prod_{\square \in \la} y_{n + c(\square)}, \]
and so $\tilde{Y_n}(\la) = 0$ unless $\forall \square \in \la$, $n + c(\square) > 0$ or $c(\square) > -n$. If $\tilde{Y_n}(\la) \not = 0$ then
	\[ \tilde{Y_n}(\la) = \prod_{\square \in \la} y_{n + c(\square)} = \prod_{\square \in \la} \frac{1}{n + c(\square)}. \]
Thus,
	\[ \tilde{I_n} = \left( \prod_{i = 1}^{n-1} i! \right)^{-1} I_n, \]
and hence, via Theorem~\ref{contentsolution}, $\{ \left( \prod_{i = 1}^{n-1} i! \right)^{-1} I_n \}_{n \in \bbZ}$ is a solution to the 2-Toda hierarchy
where $I_n = 0$ if $n < 0$. Note that a similar result appears in \cite{GGN1} with a slightly different proof. Also, this result appears in
\cite{Orlov1} from a different perspective and can be thought of as a generalization of a result of Zinn-Justin \cite{ZJ1}.

\bibliographystyle{plain}
\bibliography{TodaSolution}

\begin{thebibliography}{10}

\bibitem{AM3}
M.~Adler and P.~van Moerbeke.
\newblock The spectrum of coupled random matrices.
\newblock {\em Ann. of Math. (2)}, 149(3):921--976, 1999.

\bibitem{AM2}
M.~Adler and P.~van Moerbeke.
\newblock Hermitian, symmetric and symplectic random ensembles: {PDE}s for the
  distribution of the spectrum.
\newblock {\em Ann. of Math. (2)}, 153(1):149--189, 2001.

\bibitem{BGR1}
Edward~A. Bender, Zhicheng Gao, and L.~Bruce Richmond.
\newblock The map asymptotics constant {$t_g$}.
\newblock {\em Electron. J. Combin.}, 15(1):Research paper 51, 8, 2008.

\bibitem{Bor1}
Alexei Borodin.
\newblock Discrete gap probabilities and discrete {P}ainlev\'e equations.
\newblock {\em Duke Math. J.}, 117(3):489--542, 2003.

\bibitem{BO1}
Alexei Borodin and Grigori Olshanski.
\newblock Point processes and the infinite symmetric group.
\newblock {\em Math. Res. Lett.}, 5(6):799--816, 1998.

\bibitem{BO2}
Alexei Borodin and Grigori Olshanski.
\newblock {$Z$}-measures on partitions, {R}obinson-{S}chensted-{K}nuth
  correspondence, and {$\beta=2$} random matrix ensembles.
\newblock In {\em Random matrix models and their applications}, volume~40 of
  {\em Math. Sci. Res. Inst. Publ.}, pages 71--94. Cambridge Univ. Press,
  Cambridge, 2001.

\bibitem{BO3}
Alexei Borodin and Grigori Olshanski.
\newblock {$Z$}-measures on partitions and their scaling limits.
\newblock {\em European J. Combin.}, 26(6):795--834, 2005.

\bibitem{CG1}
S.~Carrell and I.~Goulden.
\newblock Symmetric functions, codes of partitions and the { KP } hierarchy.
\newblock {\em Journal of Algebraic Combinatorics}, 32:211--226, 2010.
\newblock 10.1007/s10801-009-0211-2.

\bibitem{E1}
N.~M. Ercolani.
\newblock Caustics, counting maps and semi-classical asymptotics.
\newblock {\em Nonlinearity}, 24(2):481--526, 2011.

\bibitem{GGN1}
I.~P. {Goulden}, M.~{Guay-Paquet}, and J.~{Novak}.
\newblock {Monotone Hurwitz numbers and the HCIZ integral I}.
\newblock {\em arXiv:1107.1015 [math.CO]}, July 2011.

\bibitem{GGN2}
I.~P. {Goulden}, M.~{Guay-Paquet}, and J.~{Novak}.
\newblock {Monotone Hurwitz numbers and the HCIZ integral II}.
\newblock {\em arXiv:1107.1001 [math.CO]}, July 2011.

\bibitem{GJ1}
I.~P. Goulden and D.~M. Jackson.
\newblock The {KP} hierarchy, branched covers, and triangulations.
\newblock {\em Adv. Math.}, 219(3):932--951, 2008.

\bibitem{HJS1}
J.~T. {Hird}, N.~{Jing}, and E.~{Stitzinger}.
\newblock {Codes and shifted codes of partitions}.
\newblock {\em arXiv:1010.4072 [math.CO]}, October 2010.

\bibitem{Hur1}
A.~Hurwitz.
\newblock Ueber {R}iemann'sche {F}l\"achen mit gegebenen {V}erzweigungspunkten.
\newblock {\em Math. Ann.}, 39(1):1--60, 1891.

\bibitem{Jo1}
Kurt Johansson.
\newblock Discrete orthogonal polynomial ensembles and the {P}lancherel
  measure.
\newblock {\em Ann. of Math. (2)}, 153(1):259--296, 2001.

\bibitem{KvL1}
V.~G. Kac and J.~W. van~de Leur.
\newblock The {$n$}-component {KP} hierarchy and representation theory.
\newblock {\em J. Math. Phys.}, 44(8):3245--3293, 2003.

\bibitem{Kaz1}
M.~Kazarian.
\newblock K{P} hierarchy for {H}odge integrals.
\newblock {\em Adv. Math.}, 221(1):1--21, 2009.

\bibitem{LZ1}
Sergei~K. Lando and Alexander~K. Zvonkin.
\newblock {\em Graphs on surfaces and their applications}, volume 141 of {\em
  Encyclopaedia of Mathematical Sciences}.
\newblock Springer-Verlag, Berlin, 2004.
\newblock With an appendix by Don B. Zagier, Low-Dimensional Topology, II.

\bibitem{Mac}
I.~G. Macdonald.
\newblock {\em Symmetric functions and {H}all polynomials}.
\newblock Oxford Mathematical Monographs. The Clarendon Press Oxford University
  Press, New York, second edition, 1995.
\newblock With contributions by A. Zelevinsky, Oxford Science Publications.

\bibitem{Ok1}
Andrei Okounkov.
\newblock Toda equations for {H}urwitz numbers.
\newblock {\em Math. Res. Lett.}, 7(4):447--453, 2000.

\bibitem{Ok2}
Andrei Okounkov.
\newblock Infinite wedge and random partitions.
\newblock {\em Selecta Math. (N.S.)}, 7(1):57--81, 2001.

\bibitem{Orlov1}
A.~Y. {Orlov}.
\newblock {Tau Functions and Matrix Integrals}.
\newblock {\em arXiv:math-ph/0210012}, October 2002.

\bibitem{Orlov2}
A.~Yu. Orlov.
\newblock Soliton theory, symmetric functions and matrix integrals.
\newblock {\em Acta Appl. Math.}, 86(1-2):131--158, 2005.

\bibitem{Orlov3}
A.~Yu. Orlov.
\newblock Hypergeometric functions as infinite-soliton tau functions.
\newblock {\em Teoret. Mat. Fiz.}, 146(2):222--250, 2006.

\bibitem{OS1}
A.~Yu. Orlov and D.~M. Shcherbin.
\newblock Hypergeometric solutions of soliton equations.
\newblock {\em Teoret. Mat. Fiz.}, 128(1):84--108, 2001.

\bibitem{Pand1}
R.~Pandharipande.
\newblock The {T}oda equations and the {G}romov-{W}itten theory of the
  {R}iemann sphere.
\newblock {\em Lett. Math. Phys.}, 53(1):59--74, 2000.

\bibitem{Ra1}
Igor Raumanov.
\newblock The correspondence between {T}racy-{W}idom and {A}dler-{S}hiota-van
  {M}oerbeke approaches in random matrix theory: the {G}aussian case.
\newblock {\em J. Math. Phys.}, 49(4):043503, 16, 2008.

\bibitem{TW3}
Craig~A. Tracy and Harold Widom.
\newblock Level-spacing distributions and the {A}iry kernel.
\newblock {\em Comm. Math. Phys.}, 159(1):151--174, 1994.

\bibitem{TW2}
Craig~A. Tracy and Harold Widom.
\newblock On orthogonal and symplectic matrix ensembles.
\newblock {\em Comm. Math. Phys.}, 177(3):727--754, 1996.

\bibitem{TW1}
Craig~A. Tracy and Harold Widom.
\newblock Correlation functions, cluster functions, and spacing distributions
  for random matrices.
\newblock {\em J. Statist. Phys.}, 92(5-6):809--835, 1998.

\bibitem{ZJ1}
P.~Zinn-Justin.
\newblock H{CIZ} integral and 2{D} {T}oda lattice hierarchy.
\newblock {\em Nuclear Phys. B}, 634(3):417--432, 2002.

\end{thebibliography}

\end{document}